\documentclass[12pt,a4paper]{article}
\usepackage{template}
\usepackage{hyperref}
\usepackage{mathtools}
\usepackage{adjustbox}
\usepackage[
	color=white,
	linecolor=gray,
	textwidth=1.8cm,
]{todonotes}

\def\calA{{\mathcal{A}}}

\def\calB{{\mathcal{B}}}

\def\calC{{\mathcal{C}}}

\def\calD{{\mathcal{D}}}

\def\calF{{\mathcal{F}}}
\def\bbF{{\mathbb{F}}}

\def\bbN{{\mathbb{N}}}
\def\calO{{\mathcal{O}}}

\def\calT{{\mathcal{T}}}

\def\bbZ{{\mathbb{Z}}}

\def\id{{\operatorname{id}}}

\def\im{{\operatorname{im}}}

\def\Hom{{\operatorname{Hom}}}
\def\Ob{{\operatorname{Ob}}}

\def\Aut{{\operatorname{Aut}}}
\def\pr{{\operatorname{pr}}}

\def\girth{{\operatorname{girth}}}
\def\diam{{\operatorname{diam}}}

\def\Cay{{\operatorname{Cay}}}

\DeclareMathOperator*{\colim}{colim}

\def\nd{{\ensuremath{^\text{nd}}}}

\def\comp{{\,\circ\,}}            
\def\pt{\mathsf{pt}}              
\def\bd{\text{bound.}}            
\def\ls{{\gg\!}}                  

\makeatletter
\def\underbrace#1{%
   \@ifnextchar_{\tikz@@underbrace{#1}}{\tikz@@underbrace{#1}_{}}}
\def\tikz@@underbrace#1_#2{%
   \tikz[baseline=(a.base)] {\node[inner sep=1] (a) {\(#1\)};
   \draw[line cap=round, line width = 0.75pt,decorate,decoration={brace,amplitude=5pt}]
     (a.south east) -- node[below,inner sep=4pt] {\(\scriptstyle #2\)} (a.south west);}}

\makeatletter
\def\overbrace#1{%
   \@ifnextchar_{\tikz@@overbrace{#1}}{\tikz@@overbrace{#1}_{}}}
\def\tikz@@overbrace#1_#2{%
   \tikz[baseline=(a.base)] {\node[inner sep=1] (a) {\(#1\)};
   \draw[line cap=round, line width = 0.75pt,decorate,decoration={brace,amplitude=5pt}]
     (a.north west) -- node[above,inner sep=4pt] {\(\scriptstyle #2\)} (a.north east);}}

\newcommand{\myrightleftarrows}[1]{\mathrel{\substack{\xrightarrow{#1} \\[-.9ex] \xleftarrow{#1}}}}

\newcommand{\set}[2]{\left\{ #1 \,\middle|\, #2 \right\}}

\newcommand{\bigslant}[2]{{\left.\raisebox{.2em}{$#1$}\middle/\raisebox{-.2em}{$#2$}\right.}}

\def\cop{{\scalebox{1.0}[0.9]{$\amalg$}}}

\makeatletter
\newcommand{\oset}[3][0ex]{%
  \mathrel{\mathop{#3}\limits^{
    \vbox to#1{\kern-4\ex@
    \hbox{$\scriptstyle#2$}\vss}}}}
\makeatother

\makeatletter
\newcommand{\osett}[3][0ex]{%
	\mathrel{\mathop{#3}\limits^{
			\vbox to#1{\kern-6\ex@
				\hbox{$\scriptstyle#2$}\vss}}}}
\makeatother

\makeatletter
\newcommand{\uset}[3][0ex]{%
  \mathrel{\mathop{#3}\limits^{
    \vbox to#1{\kern8\ex@
    \hbox{$\scriptstyle#2$}\vss}}}}
\makeatother

\def\VCyc{{\mathcal{VC}\textit{yc}}}
\def\Fin{{\mathcal{F}\textit{in}}}

\def\O{\calO}

\def\T{\calT}

\def\D{\calD}

\def\lf{\textit{lf}}

\author{Markus Zeggel}
\title{The Bounded Isomorphism Conjecture\\ for Box Spaces of Residually Finite Groups}

\begin{document}
	
	\maketitle
	
	\paragraph*{Abstract}
	In this article we study a coarse version of the $K$-theoretic Farrell--Jones conjecture we call \emph{coarse} or \emph{bounded isomorphism conjecture}.
	Using controlled category theory we are able to translate this conjecture for asymptotically faithful covers into a more familiar form.
	This allows us to prove the conjecture for box spaces of residually finite groups whose Farrell--Jones assembly map with coefficients is an isomorphism.
	
	\tableofcontents
	
	\section{Introduction}

There are two important isomorphism conjectures in the fields of operator algebras and geometric topology: the Baum--Connes conjecture and the Farrell--Jones conjecture. In the Davis--Lück picture, these conjectures predict the following: For any group $G$, the assembly maps 
$$H^G_*(E_\text{Fin} G; \mathbb K^\text{top}) \rightarrow K_*(C^*_r(G))
\quad \text{ and } \quad
H^G_*(E_\text{VCyc} G; \mathbb K_R^\text{alg}) \rightarrow K_*(R[G])$$
are isomorphisms. These formulae help compute the topological $K$-theory of the reduced group $C^*\!$-algebra $C^*_r(G)$ and the algebraic $K$-theory of the group ring $R[G]$, respectively.

In order to prove or disprove such conjectures, it can be helpful to compare them to other versions of the conjectures.
For example, there is a coarse version of the Baum-Connes conjecture.
Using the connection between the ``coarse'' and the ``usual'' conjecture, one can show that the Baum--Connes assembly map fails to be surjective for Gromov monster groups, i.e.\ groups that coarsely contain infinite expanders, cf.\ \cite[Theorem~8.2]{Willett.2012a}.
Such methods of constructing counter-examples are based on the work of Higson, Lafforgue and Skandalis, see \cite{Higson.2002}.

It is natural to ask whether the coarse version of the Farrell-Jones conjecture behaves differently.
One way of constructing infinite expanders is to take the Cayley graphs of the groups $SL_n(\bbF_p)$ for an increasing sequence of prime numbers $p$.
All these groups are finite quotients of the residually finite group $SL_n(\bbZ)$.
Instead of directly studying expanders, as in \cite{Willett.2012a}, we here study residually finite groups.

\subsection{Result}

The reduced and the maximal version of the Baum--Connes conjecture behave very differently.
As we already indicated, box spaces of certain residually finite groups yield counter-examples for the reduced Baum--Connes conjecture.
In contrast, Oyono-Oyono and Yu prove the maximal coarse Baum--Connes conjecture for box spaces of residually finite groups that satisfy the maximal Baum--Connes conjecture with coefficients, see  \cite[Theorem 4.17]{Oyono-Oyono.2009}.
In this article we prove the analogous result for the coarse algebraic assembly map.

\begin{theorem*}
	Let $X$ be a box space of a residually finite group $G$, i.e.\ $X = \bigsqcup_{n \in \bbN} G/H_n$, where $H_n$ is a sequence of finite index normal groups in $G$ with $\bigcap_{n \in \bbN} H_n = \{1\}$.
	Then the coarse assembly maps combine to an isomorphism
	$$\colim_{d \geq 0} H_*^\lf\big(P_d(X); \calA\big) \rightarrow K_*\big(\calC^\lf( X; \calA)\big)$$
	if and only if $G$ satisfies the Farrell--Jones conjecture with coefficients $\calA_G$. Here, $\calA_G$ is an additive category constructed from $\calA$ and the sequence $H_n$, see Theorem \ref{thm:reduction_for_residually_finite_groups} for details.
\end{theorem*}

\subsection{Outline}

In Section \ref{sec:basics} we recall basic concepts for additive categories and introduce so-called $\varepsilon$-filtered categories.
Using these notions, we define different kinds of assembly maps and compare them to each other in Section \ref{sec:assembly_maps}.
In Section \ref{sec:the_conjecture} we give a definition of the bounded isomorphism conjecture and translate it into another form.
This translation already makes use of the properties of residually finite groups and works more generally for so-called asymptotically faithful covers.
Finally, we establish the relation between the coarse and the usual assembly map for residually finite groups as stated in the theorem above.
We conclude this article with an overview of some counterexamples for the coarse Baum--Connes conjecture that are proexamples for the coarse Farrell--Jones conjecture.

\subsubsection*{Acknowledgements}

This article is based of my PhD thesis. I would like to thank my supervisor Arthur Bartels and my colleague Robin Loose for many helpful discussions.

The project was funded by the Deutsche Forschungsgemeinschaft (DFG, German Research Foundation) -- Project-ID 427320536 -- SFB 1442, as well as under Germany’s Excellence Strategy EXC 2044 390685587, Mathematics Münster: Dynamics--Geometry--Structure.
	\section{Basics}
\label{sec:basics}

\subsection{Additive Categories}

\begin{definition}
	Let $\calC$ be an additive category and $\calA \subseteq \calC$ a full additive subcategory. The \emph{quotient category} $\bigslant\calC\calA$ is the category that has the same objects as $\calC$, but
	$$\Hom_{\calC/\calA}(X,Y) := \bigslant{\Hom_\calC(X,Y)}{\sim}$$
	as morphism groups, where $f \sim g$ if and only if $f-g$ factors through an object in $\calA$, i.e.\ if there exists a commutative diagram in $\calC$ \vspace{-1em}
	\begin{center}
		\begin{tikzcd}[row sep=1em, column sep=1.5em]
			X \arrow[rr, "f-g"] \arrow[dr] && Y \\ & A \arrow[ur]
		\end{tikzcd}
	\end{center}
	with $A \in \calA$. This defines an additive category via $[f] \comp [g] := [f \comp g]$ and $[f] + [g] := [f + g]$. 
	Also, the canonical functor $\pr: \calC \rightarrow \bigslant{\calC}{\calA}$ is compatible with direct sums (cf. \cite{Ramras.2018}).
\end{definition}

\begin{definition}
	\label{def:karoubi}
	Let $\calC$ be an additive category and $\calA \subseteq \calC$ a full additive subcategory. 
	The pair $(\calC, \calA)$ is called a \emph{Karoubi filtration} if, for all $A \in \calA$, $X \in \calC$ and $(f: A \rightarrow X)$, $(g: X \rightarrow A) \in \Hom_\calC$, there are a direct sum decomposition $A' \underset{\adjustbox{trim=0ex 0ex 0ex 1.25ex}{{\scriptsize $\pr$}}}{\oset{\adjustbox{trim=0ex 1.75ex}{{\scriptsize $\iota$}}}{\myrightleftarrows{~~~}}} X \myrightleftarrows{~~~} Y$, with $A' \in \calA$, $Y \in \calC$, and morphisms $f',g'$
	such that \vspace{-1em}
	\begin{center}
		\begin{tikzcd}[row sep=1.5em, column sep=1.5em]
			A \arrow[rr, "f"] \arrow[dr, swap, "f'"] && X \\ & A' \arrow[ur, swap, "\iota"]
		\end{tikzcd} and 
		\begin{tikzcd}[row sep=1.5em, column sep=1.5em]
			X \arrow[rr, "g"] \arrow[dr, swap, "\pr"] && A \\ & A' \arrow[ur, swap, "g'"]
		\end{tikzcd}
	\end{center}
	are commutative diagrams in $\calC$. 
	(Note that it is the same $A'$ in both diagrams.)
\end{definition}

\begin{remark}
	The original definition of a Karoubi filtration is slightly different.
	However, in \cite[Lemma 5.6 and Remark 5.7]{Kasprowski.2014} the definition we use is shown to be equivalent to the original definition.
\end{remark}

\subsection{$\varepsilon$-Filtered Categories}

With the following definition we are able to apply concepts of basic analysis to category theory.

\begin{definition}
	We call an additive category $\calA$ an \emph{$\varepsilon$-filtered category} if, for every $\varepsilon > 0$ and pair of objects $X,Y \in \calA$, there is a collection of (additive) subgroups $\Hom_\varepsilon(X, Y) \subseteq \Hom(X,Y)$, called \emph{the subgroup of $\varepsilon$-controlled morphisms}, such that
	\begin{enumerate}
		\item $\varepsilon \leq \varepsilon' \Rightarrow \Hom_\varepsilon(X,Y) \subseteq \Hom_{\varepsilon'}(X,Y)$,
		\item $\bigcup_{\varepsilon > 0} \Hom_\varepsilon(X,Y) = \Hom(X,Y)$,
		\item $\id_X \in \Hom_\varepsilon(X,X)$ for all $\varepsilon > 0$, and
		\item $f \in \Hom_\varepsilon(X,Y), g \in \Hom_{\varepsilon'}(Y,Z) \Rightarrow g \comp f \in \Hom_{\varepsilon + \varepsilon'}(X, Z)$.
	\end{enumerate}
\end{definition}

The direct sum of $\varepsilon$-filtered categories is again $\varepsilon$-filtered.
However, there is no canonical way of defining a useful $\varepsilon$-filtration on the direct product of infinitely many $\varepsilon$-filtered categories.
The following is a remedy for that.

\begin{definition}
	Let $\calA_i$, $i \in I$, be a family of $\varepsilon$-filtered categories. The \emph{bounded product} $\prod^\bd_{i \in I} \calA_i$ is the subcategory of $\prod_{i \in I} \calA_i$ that has all objects, but only includes those morphisms $(f_i)_i$ that are uniformly $\varepsilon$-controlled for an $\varepsilon > 0$, i.e.\ $f_i \in \Hom_\varepsilon$ for all $i \in I$.
\end{definition}

The following definition describes how we can create several different $\varepsilon$-filtered categories from a given metric space.

\begin{definition}
\label{def:controlled_categories}
Let $G$ be a group and $(X,d)$ a locally compact metric space with isometric $G$-action.
Let $\calA$ be an additive category with $G$-action.
The category $\O_G^\lf(X;\calA)$ is defined as follows.
\begin{description}
\item[Objects] are triples $(S, \pi: S \rightarrow X \times \bbN, M: S \rightarrow \Ob(\calA))$ where $S$ is a free $G$-set such that $\pi$ and $M$ are $G$-equivariant ($G$ acts trivially on $\bbN$) and 
\begin{enumerate}
	\item \label{c_locally_compact} for all cocompact $G$-subsets $K \subseteq X \times \bbN$ the preimage $\pi^{-1}(K)$ is cofinite, i.e.\ has finite quotient.
\end{enumerate}
\item[Morphisms] $\varphi: (S, \pi, M) \rightarrow (S', \pi', M')$ are given by matrices $(\varphi_{s'}^s)_{s \in S, s' \in S'}$ where $\varphi_{s'}^s: M(s) \rightarrow M'(s')$ is a morphism in $\calA$ such that $g.\varphi_{s'}^s = \varphi_{gs'}^{gs}$ and where
\begin{enumerate}
	\setcounter{enumi}{1}
\item \label{c_rows_columns_finite} all rows and columns are finite,
\item \label{c_bounded_n} $\exists \alpha > 0 \,\forall s, s' : \varphi_{s'}^s \neq 0 \Rightarrow |\pi_\bbN(s) - \pi'_\bbN(s')| < \alpha $,
\item \label{c_bounded_x} $\exists \alpha > 0 \,\forall s, s' : \varphi_{s'}^s \neq 0 \Rightarrow d(\pi_X(s), \pi'_X(s')) < \alpha$, and
\item \label{c_controlled} $\forall \varepsilon>0 \,\exists t_0>0 \,\forall s,s': \pi_\bbN(s) > t_0 \wedge \varphi_{s'}^s \neq 0 \Rightarrow d(\pi_X(s), \pi'_X(s')) < \varepsilon$.
\end{enumerate}
Composition is given by matrix multiplication, i.e.\ $(\varphi \comp \psi)^s_{s''} := \sum\limits_{s'} \varphi^{s'}_{s''} \comp \psi^s_{s'}$.
\end{description}
A morphism $\varphi: (S, \pi, M) \rightarrow (S', \pi', M')$ in this category is \emph{$\varepsilon$-controlled} if 
$$\forall s \in S, s' \in S': \varphi_{s'}^s \neq 0 \Rightarrow d(\pi_X(s), \pi'_X(s')) < \varepsilon \wedge |\pi_\bbN(s) - \pi'_\bbN(s')| < \varepsilon.$$
This category can be modified in various ways. 
For instance, there are further properties that morphisms can fulfil:
\begin{enumerate}
\setcounter{enumi}{5}
\item \label{c_concentrated_morphisms} $\forall s \in S, s'\in S': \varphi_{s'}^s \neq 0 \Rightarrow \pi_X(s) = \pi'_X(s')$.
\end{enumerate}
If we want the category to satisfy Property \ref{c_concentrated_morphisms}, we add a subscript 0 to the name, i.e.\ $\calO_{G,0}^\lf$.
There are also properties that put further restrictions on the objects. We list them here and add the modifications to the notation in brackets.
\begin{enumerate}
\setcounter{enumi}{6}
\item \label{c_compact} $\exists K \subseteq X \text{ cocompact}: \im(\pi_X) \subseteq K$. \hfill (Drop $^\lf$ from the notation.)
\item \label{c_finite_range} $\im(\pi_\bbN)$ is finite.\hfill (Use the letter $\calT$ instead of $\calO$.)
\item \label{c_only_space} $\im(\pi_\bbN) \subseteq \{0\}$.\hfill (Use the letter $\calC$ instead of $\calO$.)
\end{enumerate}
If $G$ is the trivial group, we also drop it from the notation.
Note that Property~\ref{c_only_space} implies Property~\ref{c_finite_range}. This leaves twelve possible variations of the category of which we will use six within this paper. The names of these variations are listed here. The ones we will not use have been greyed out.
$$
\calO^\lf \;\quad\;
{\color{lightgray} \calO^\lf_0} \;\quad\;
\calO \;\quad\;
{\color{lightgray} \calO_0} \;\quad\;
\calT^\lf \;\quad\;
{\color{lightgray} \calT^\lf_0} \;\quad\;
\calT \;\quad\;
{\color{lightgray} \calT_0} \;\quad\;
\calC^\lf \;\quad\;
{\color{lightgray} \calC^\lf_0} \;\quad\;
{\color{lightgray} \calC} \;\quad\;
\calC_0
 $$
Note also that Property \ref{c_controlled} is only necessary for the categories $\calO^\lf$ and $\calO$, as in all other cases it is already implied by Property \ref{c_concentrated_morphisms}, \ref{c_finite_range} or \ref{c_only_space}.
\end{definition}

\begin{lemma}
	\label{lem:karoubi_filtration}
	For all locally finite metric spaces $X$ with an isometric $G$-action,  $\big(\calO^\lf(X; \calA), \calT^\lf(X; \calA)\big)$, $\big(\calO_G(X; \calA), \calT_G(X; \calA)\big)$, $\big(\calO^\lf(X; \calA), \calO(X; \calA)\big)$, and $\big(\calT^\lf(X; \calA), \calT(X; \calA)\big)$ yield Karoubi filtrations.
\end{lemma}

\begin{proof}
	Essentially, these are only two cases: the $(\calO, \calT)$-case and the $(-^\lf, -)$-case.
	In either case we consider morphisms
	$$\varphi: (T, \rho, N) \rightarrow (S, \pi, M) \quad \text{ and } \quad \psi: (S, \pi, M) \rightarrow (T, \rho, N).$$
	In both cases we find a direct summand $(S', \pi|_{S'}, M|_{S'})$ of $(S, \pi, M)$, with $S' \subseteq S$, which is also an object in the according subcategory, together with the following factorizations.
	\begin{center}
		\begin{tikzcd}[row sep=1.5em, column sep=-1em]
			(T, \rho, N) \arrow[rr, "\varphi"] \arrow[dr, swap, "\varphi"] && (S, \pi, M) \\
			& (S', \pi|_{S'}, M|_{S'}) \arrow[ur, swap, "\iota"]
		\end{tikzcd} $\quad$ 
		\begin{tikzcd}[row sep=1.5em, column sep=-1em]
			(S, \pi, M) \arrow[rr, "\psi"] \arrow[dr, swap, "\pr"] && (T, \rho, N) \\
			& (S', \pi|_{S'}, M|_{S'}) \arrow[ur, swap, "\psi"]
		\end{tikzcd}
	\end{center}
	In the $(\calO, \calT)$-case there is an $L \in \bbN$ such that $\im(\pi_{\bbN}) \subseteq \{1, \dots, L\}$.
	We choose $\alpha > 0$ such that Property \ref{c_bounded_n} is fulfilled and set $S' := \pi_{\bbN}^{-1}(1, \dots, L + \alpha)$.
	
	In the $(-^\lf, -)$-case there is a compact set $K \subseteq X$ such that $\im(\pi_X) \subseteq K$.
	We choose $\alpha > 0$ such that Property \ref{c_bounded_x} is satisfied for both $\varphi$ and $\psi$ and set $S' := \pi_{X}^{-1}(B_\alpha(K))$, where $B_\alpha(K)$ is the closed $\alpha$-neighbourhood of $K$.
\end{proof}

\begin{lemma}
	\label{lem:O_is_trivial}
	For every additive category $\calA$ the $K$-Theory of the category $\calO_G(\pt; \calA)$ (with decorations) is trivial.
\end{lemma}

\begin{proof}
	An Eilenberg swindle is an endofunctor $F$ with the property $F + \id \simeq F$.
	Since any two naturally isomorphic functors induce the same map on $K$-theory, it follows that $F(\id) = 0$.
	We define an Eilenberg swindle on $\calO_G(\pt; \calA)$ as follows.
	$$F^n: (S, \pi, M) \mapsto (S, \pi^n, M) \quad \text{ with } \quad \pi^n(s) = (\pi_\pt(s), \pi_\bbN(s) + n)$$
	$$F := \sum_{n \in \bbN} F^n : \calO_G(\pt; \calA) \rightarrow \calO_G(\pt; \calA)$$
	This functor is well-defined because for $\pt$ Property \ref{c_controlled} is trivially satisfied.
\end{proof}

\begin{definition}[Group Ring Construction]
	The additive category $\calA[G]$ has the same objects as $\calA$ and the following morphisms
	$$\hom_{\calA[G]}(X,Y) = \set{\sum\nolimits_{g \in G} \varphi_g \; g}{\varphi_g : g.X \rightarrow Y}.$$
	In the formal sum we only allow a finite number of summands to be non-trivial.
\end{definition}

\begin{lemma}
	\label{lem:group_ring_construction}
	For any group $G$, non-empty space $X$ and additive category $\calA$, there is a natural equivalence of categories
	$$\calT_G(X; \calA) \simeq \calA[G].$$
\end{lemma}

\begin{proof}
	Choose a base point $x_0 \in X$.
	We define a functor $\calA[G] \rightarrow \calT_G(X; \calA)$ via $A \mapsto (G,\, g \mapsto g.x_0,\, g \mapsto g.A)$.
	On morphisms we map as follows
	$$\left[\sum_{g \in G} \varphi_g \; g, \;\; \varphi_g : g.A \rightarrow B\right] \mapsto \varphi: (G, g.x_0, g.A) \rightarrow (G, g.x_0, g.B),$$
	where $\varphi$ is determined by $ \varphi_{1}^g := \varphi_g : g.A \rightarrow B$.
	Clearly, the functor is fully faithful.
	It is also essentially surjective, since any object $(S, \pi, M) \in \calT_G(X; \calA)$ is isomorphic to the image of $\bigoplus_{s \in S'} M(s)$, where $S' \subseteq S$ is a set of representatives for the $G$-orbits of $S$.
\end{proof}
	\section{Assembly Maps}
\label{sec:assembly_maps}

Morphisms in the categories $\calO^\lf(X)$ and $\calO(X)$ have the property of becoming arbitrarily short when moving to infinity along the $\bbN$-direction.
If we could ignore the part that is far away from infinity, then these categories would only detect local information of the space $X$.
This can be achieved by forming the quotient category $\calO(X) / \calT(X)$.

To make notation simpler we abbreviate
$$\calD_G(X; \calA) := \bigslant{\calO_G(X; \calA)}{\calT_G(X; \calA)}$$
and
$$\calD^\lf(X; \calA) := \bigslant{\calO^\lf(X; \calA)}{\calT^\lf(X; \calA)}.$$
It is a fact (cf.\ \cite[Section 5]{Bartels.2005} and \cite[Section 7.2]{Bartels.2007b}) that the functor
$$ X \mapsto H^G_n(X; \calA) := K_{n+1}\big(\calD_G(X; \calA)\big) $$
yields a $G$-equivariant homology theory with $H^G_n(G/H; \calA) = K_n(\calA[H])$.
Similarly, the functor
$$ X \mapsto H^\lf_n(X; \calA) := K_{n+1}\big(\calD^\lf(X; \calA)\big) $$
yields a locally finite homology theory with $H^\lf_n(\pt; \calA) = K_n(\calA)$, cf.\ \cite{Weiss.2002}.
Clearly, both definitions have many similarities, e.g.\ $H_*^\lf(X; \calA) = H_*(X; \calA)$ for compact spaces $X$.
A very useful aspect of these definitions is the existence of long exact Karoubi sequences (cf.\ \cite{Cardenas.1995})
$$ \dots \rightarrow
K_n\big(\calT_G(X)\big) \rightarrow
K_n\big(\calO_G(X)\big) \rightarrow
K_n\big(\calD_G(X)\big) \xrightarrow{\partial}
K_{n-1}\big(\calT_G(X)\big) \rightarrow \dots$$
and
$$ \dots \rightarrow
K_n\big(\calT^\lf(X)\big) \rightarrow
K_n\big(\calO^\lf(X)\big) \rightarrow
K_n\big(\calD^\lf(X)\big) \xrightarrow{\partial}
K_{n-1}\big(\calT^\lf(X)\big) \rightarrow \dots$$
which we have due to Lemma \ref{lem:karoubi_filtration}.

\subsection{The Farrell--Jones Conjecture}

Let $\calA$ be an additive category and $G$ a group.
Let $E$ be a metric space with an isometric $G$-action. The canonical map $E \rightarrow \pt$ induces a map \vspace{-0.5ex}
$$\alpha : H_*^G(E; \calA) \rightarrow H_*^G(\pt; \calA) \cong K_*\big(\calA[G]\big).$$
We are particularly interested in the space $E = E_\calF G$, where $\calF$ is a family of subgroups of $G$.
The map $\alpha$ is called \emph{assembly map}.

\begin{conjecture}[Farrell--Jones conjecture with coefficients]
The assembly map
$$\alpha_\VCyc : H_*^G(E_\VCyc G; \calA) \rightarrow K_*\big(\calA[G]\big)$$
is an isomorphism for all groups $G$ and additive categories $\calA$.
Here $\VCyc$ is the family of virtually cyclic subgroups of $G$.
\end{conjecture}

\begin{remark}
\label{rem:assembly_map_vs_connecting_hom}
The following diagram commutes by naturality of $\partial$.
\begin{center}
\begin{tikzcd}[column sep=-2.9ex]
H_*^G(E; \calA) =\;\;\, & K_{* + 1}\big(\D_G(E; \calA)\big)  \arrow[rr] \arrow[d, "\partial"] & \hspace{1.5cm} & K_{* + 1}\big(\D_G(\pt; \calA)\big) \arrow{d}{\partial}[swap]{\cong} \\
& K_*\big(\T_G(E; \calA)\big) \arrow[rr, "\cong"] && K_*\big(\calT_G(\pt; \calA)\big) & \cong K_*\big(\calA[G]\big)
\end{tikzcd}
\end{center}
The connective homomorphism on the right hand side is an isomorphism because $K_*(\calO_G(\pt; \calA)) = 0$, see Lemma \ref{lem:O_is_trivial}. The lower map is an isomorphism because $\calT_G(-; \calA)$ is naturally equivalent to the constant functor $\calA[G]$, see Lemma \ref{lem:group_ring_construction}.
Thus, in order to study the assembly map, we might as well study the connective homomorphism \vspace{-0.5ex}
$$\partial: K_{* + 1}\big(\D_G(E; \calA)\big) \rightarrow K_*\big(\T_G(E; \calA)\big),$$
which will be our point of view from now on.
\end{remark}

\subsection{Motivation for a Coarse Conjecture}

We have the following commutative diagram, in which the vertical arrows are induced by the forgetful maps. \vspace{-1em}
\begin{center}
\begin{tikzcd}
K_{*+1}\big(\D_G(E; \calA)\big) \arrow[r, "\partial"] \arrow[d] & K_*\big(\T_G(E; \calA) \big) \arrow[d] \\
K_{*+1}\big(\D^\lf(E; \calA)\big) \arrow[r, "\partial^\lf"] & K_*\big(\T^\lf(E; \calA) \big)
\end{tikzcd}
\end{center}
More conceptually, this diagram can be written as
\begin{center}
\begin{tikzcd}
H_*^G(E; \calA) \arrow[r, "\alpha"] \arrow[d] & K_*\big(\calA[G]\big) \arrow[d] \\
H_*^\lf(E; \calA) \arrow[r, "\alpha_c"] & K_*\big(\calC^\lf(E; \calA) \big)
\end{tikzcd}
\end{center}
in which we have exchanged the categories on the right hand side by equivalent ones.
In particular, we make use of the facts that $\calT_G(E; \calA) \simeq \calA[G]$ and $\calT^\lf(E; \calA) \simeq \calC^\lf(E; \calA)$.
(Note that, on the left hand side, we merely unravelled the definition of $H^G$ and $H^\lf$.)
The upper arrow is the the Farrell--Jones assembly map.
The lower one is called \emph{coarse} or \emph{bounded assembly map}.

If the homomorphism $H_*^G(E; \calA) \rightarrow H_*^\lf(E; \calA)$ induced by the forgetful functor were to be an isomorphism (or at least injective), then injectivity of $\alpha_c$ would imply injectivity of $\alpha$.
This is merely a very rough sketch of how the usual and the coarse assembly map are connected.
Since the forgetful map above is not an isomorphism, we have to help ourselves with a version of the commutative square that involves homotopy fixed points.
This is nicely written up in \cite[Section 3.1]{Kasprowski.2014}.
Gaining information about the assembly map like this is called \emph{descent principle} which goes back to Carlsson and Pedersen, \cite{Carlsson.Pedersen.1995}.

It is natural to ask whether the following conjecture is true.

\begin{conjecture}[{Cf.\ \cite[Conjecture 6.28]{Roe.1993}}]
\label{conj:uniformly_contractible_conjecture}
The coarse assembly map 
$$ \alpha_c : H_*^\lf\big(E; \calA\big) \rightarrow K_*\big(\calC^\lf(E; \calA)\big)$$
is an isomorphism for all uniformly contractible spaces $E$.
\end{conjecture}

By replacing $H^\lf_*(-; \calA)$ by $\colim_{d \geq 0} H^\lf_*(P_d(-); \calA)$, Conjecture \ref{conj:uniformly_contractible_conjecture} can be modified such that we do not have to assume $E$ to be uniformly contractible.
Here, $P_d(X)$ is the \emph{Rips complex} of a metric space $X$.

\begin{definition}
	Let $X$ be a metric space and $d \geq 0$. The \emph{Rips complex} of $X$ is the simplicial complex $P_d(X)$ defined as follows: A finite subset $\sigma \subseteq X$ is a simplex if $\diam(\sigma) \leq d$.
\end{definition}

The assignment $X \mapsto \colim_{d \geq 0} H^\lf_*(P_d(X); \calA)$ is functorial w.r.t.\ coarse maps and a simple homotopy lifting argument shows that any two coarsely equivalent maps induce the same map.

\begin{conjecture}
	\label{conj:main_conjecture}
	Let $X$ be a locally compact metric space. 
	The coarse assembly maps combine to an isomorphism
	$$\colim_{d \geq 0} H_*^\lf\big(P_d(X); \calA\big) \rightarrow K_*\big(\calC^\lf( X; \calA)\big).$$
	More precisely, this morphism is induced by the following compositions
	$$H_*^\lf\big(P_d(X); \calA\big) \xrightarrow{\;\alpha_c\;} K_*\big(\calC^\lf(P_d(X); \calA)\big) \xrightarrow{\;\cong\;} K_*\big(\calC^\lf(X; \calA)\big).$$
	Here, the latter map is an isomorphism because the inclusion $X \rightarrow P_d(X)$ is a coarse equivalence.
\end{conjecture}

A conjecture like this is typically proven by showing that the third term in a long exact sequence vanishes.
In our case, for instance, we could prove\footnote{This proof uses the long exact sequence for the Karoubi filtration $\big(\calO^\lf, \calT^\lf\big)$ and the fact that directed colimits are exact. Cf.\ Corollary \ref{cor:third_term}, which is very similar.} the conjecture by showing 
$$\colim_{d \geq 0} K_*\big(\calO^\lf(P_d(X); \calA)\big) = 0.$$
This is exactly what Ramras, Tessera and Yu do in \cite{Ramras.2018}, where they prove this conjecture for spaces $X$ of finite decomposition complexity (cf. \cite[Theorem 6.4]{Ramras.2018}).
	\section{The Bounded Isomorphism Conjecture for Asymptotically Faithful Covers}
\label{sec:the_conjecture}

The goal for this section is to state Conjecture \ref{conj:main_conjecture} in an equivalent form that we are able to prove.
The idea is that the problem can be lifted along a covering and solved in a more familiar setting.

\subsection{$R$-Metric Covers}

Let us recall some basics concerning $R$-metric covers from \cite{Willett.2012a}.
Definitions \ref{def:R_metric_cover} and \ref{def:asymptotically_faithful} can be found in \cite[Section 2]{Willett.2012a}.

\begin{definition}
	\label{def:R_metric_cover}
	Let $(X,d)$, $(Y,d')$ be metric spaces, $p: X \rightarrow Y$ a surjective map and $R > 0$. The map $p$ is called an \emph{$R$-metric cover} if, for all $x \in X$, the map $p: X \rightarrow Y$ restricts to an isometry $p|_{B_R(x)}: B_R(x) \rightarrow B_R(p(x))$. We call $p: X \rightarrow Y$ a \emph{metric cover} to imply the existence of an $R > 0$ such that $p$ is an $R$-metric cover.
\end{definition}

\begin{lemma}
\label{lem:R_translative}
Let $p: X \rightarrow Y$ be a connected $R$-metric cover and $G = \Aut(p)$ its deck transformation group. The action $G \curvearrowright X$ is \emph{$R$-translative}, i.e.\
$$\forall x \in X, g \in G \smallsetminus \{1\}: d(x,gx) \geq R.$$
\end{lemma}

\begin{proof}
For all $x \in X$ and $g \in G$ we have isometries
$$B_R(x) \rightarrow B_R(p(x)) = B_R(p(gx)) \leftarrow B_R(gx).$$
Assume $d(x,gx) < R$, i.e.\ $x \in B_R(gx)$. We have
$$0 = d(p(x), p(x)) = d(p(x), p(gx)) = d(x, gx)$$
and therefore $x = gx$, which implies $g = 1$, since $G \curvearrowright X$ is free.
\end{proof}

\begin{definition}
	\label{def:asymptotically_faithful}
	Let $p_n : X_n \rightarrow Y_n$ be a sequence of metric covers. We say it is \emph{asymptotically faithful} if there is a sequence $R_n > 0$ with $R_n \uset{n \rightarrow \infty}{\longrightarrow} \infty$ such that $p_n$ is an $R_n$-metric cover.
\end{definition}

\begin{lemma} 
\label{lem:residually_finite_gives_sequence}
Let $G$ be a finitely generated residually finite group with a finite generating set $S$. This set induces a generating set $[S]$ for $G/H_n$. Then $G$ yields an asymptotically faithful sequence $p_n: (G, d_S) \rightarrow (G/H_n, d_{[S]})$ of metric covers.
\end{lemma}

\begin{proof}
Consider the sequence  $(H_n)_{n \in \bbN}$ of finite index normal subgroups with $H_{n+1} \subseteq H_n$ and $\bigcap_{n \in \bbN} H_i = \{1\}$. 
The trivial intersection implies that for every $R > 0$ and $n \gg R$, the set $S^{2R} = B_{2R}(1)$ has no intersection with $H_n$.
Let $n \gg R$ with $B_{2R}(1) \cap H_n = \emptyset$ and let $g \in G$ with $|g|_G \leq R$.
Let $[g] = [s_1] \cdots [s_m]$ in $G/H_n$ with $m$ minimal, i.e. $m = |[g]|_{[S]} \leq |g|_S \leq R$.
By definition, we have $g = s_1 \cdots s_n \cdot h$ for some $h \in H_n$. The assumption implies that $|h|_G > 2R$ if $h \neq 1$ and in that case the reverse triangle inequality gives
$$|g|_G \geq |h|_G - |s_1 \cdots s_m|_G > 2R - R = R.$$
As this is a contradiction to $|g|_G \leq R$, we conclude that $g = s_1\cdots s_m$ and therefore $|g|_S = |[g]|_{[S]}$.
Thus, $p_n: G \rightarrow G/H_n$ is an $R$-metric cover.
\end{proof}

\begin{lemma}
	\label{lem:metric_cover_gives_metric_cover_after_rips}
	Let $X$ and $Y$ be graphs and $p: X \rightarrow Y$ an $R$-metric cover. If $R \geq 3d$, then $p_*: P_d(X) \rightarrow P_d(Y)$ is an $(\lfloor \frac{R}{d} \rfloor - 1)$-metric cover.
\end{lemma}

\begin{proof}
	First we show that $p_*$ is surjective as a simplicial map. Let $\sigma \subset P_d(Y)$ be a simplex, i.e.\ $\diam(\sigma) \leq d \leq R$. Since $p$ is surjective we can pick $x \in X$ with $p(x) = y$ for some $y \in \sigma$. Now, $p|_{B_d(x)} : B_d(x) \rightarrow B_d(y)$ is an isometry, therefore $\sigma' := p|_{B_d(x)}^{-1}(\sigma)$ has diameter $\diam(\sigma') = \diam(\sigma) \leq d$ and thus is a preimage of $p_*$.
	
	Let $x \in X$ be a vertex. Any vertex $x' \in X$ within distance $r := \lfloor\frac{R}{d}\rfloor$ of $x$ in $P_d(X)$ has distance at most $R$ to $x$ in $X$. Hence, $p_*|_{B_{r}(x)}$ is an isometry. The claim follows from the fact that any point in a simplicial complex lies within distance 1 of a vertex.
\end{proof}

\begin{corollary}
	\label{cor:Rips_complex_respects_asymptotically_faithful_covers}
	Let $d \geq 1$ and $p_n: E_n \rightarrow B_n$ be an asymptotically faithful sequence of covers of graphs. Then $(p_n)_*: P_d(E_n) \rightarrow P_d(B_n)$ is an asymptotically faithful sequence of covers.
\end{corollary}

\subsection{Preparations}

\begin{notation}
	We abbreviate \mbox{$\calO^\gg := \calO^\lf / \calO$}, i.e.\ given a space $X$ and an additive category $\calA$, we write $\calO^\gg(X; \calA)$ for the category $\bigslant{\calO^\lf(X; \calA)}{\calO(X; \calA)}$.
	Similarly, given a sequence of $\varepsilon$-filtered categories $\calA_n$, we abbreviate the quotient $\bigslant{\prod^\bd_n \calA_n}{\bigoplus_n \calA_n}$ as $\frac{\prod^\bd_n}{\bigoplus_n} \calA_n$.
\end{notation}

\begin{lemma}
	\label{lem:isomorphism_to_equivariant}
	Let $p_n: E_n \rightarrow B_n$ be an asymptotically faithful sequence of covers with deck transformation groups $G_n$, where every $B_n$ is compact. 
	The functor
	$$\Phi : \frac{\prod^\bd_{n \in \bbN}}{\bigoplus_{n \in \bbN}} \calO_{G_n}(E_n; \calA_n) \rightarrow \calO^\ls\left(\coprod_{n \in \bbN} B_n; \calA_n\right)$$
	induced by the $p_n$ (see proof for the precise definition) is an equivalence of categories.
\end{lemma} 

\begin{proof}
We define $\Phi$ on objects via
$$(S_n, \pi_n, M_n)_n \mapsto \big(\coprod\nolimits_{n} S_n / G_n, \; \cop_{n} (p_n \times \id_\bbN) \comp \pi_n, \; \cop_{n} M_n\big)$$
and on morphisms via
$$(\varphi_n)_n \mapsto \widetilde \varphi \quad \text{ where } \quad \widetilde \varphi^{[s]}_{[s']} = \begin{cases}
\sum_{h \in G_n} (\varphi_n)^s_{hs'} & s, s' \in S_n \\
0 & \text{otherwise.}
\end{cases}$$
All of these definitions are well-defined because everything involved is $G_n$-equivariant and the sum is, indeed, finite. The latter is true because the rows of $\varphi_n$ are cofinite.
Note that because of the boundedness conditions, for every morphism $\varphi$ of $\calO(\coprod_{n \in \bbN} B_n; \calA_n)$, we have $\varphi_{s'}^s = 0$ if $s$ and $s'$ live over different components in $\coprod_{n \in \bbN} B_n$. 
This argument shows the fullness of $\Phi$: Over each component $E_n$, morphisms can be defined on an orbit and then be extended $G_n$-equivariantly.

\vspace{1ex}
It is easily seen that $\Phi$ is essentially surjective: Every object $(S, \pi, M) \in \calO^\lf(\coprod_{n \in \bbN} B_n; \calA_n)$ is isomorphic to the image of $(G_n \times S_n, \pi_n, M_n)_n$, where 
\begin{align*}
S_n &:= \pi_{\coprod_n B_n}^{-1}(B_n), \\
\pi_n(g,s) &:= (g.x_s, \pi_{\bbN}(s)) \quad \text{ for a fixed chosen $x_s \in p^{-1}(\pi(s)) \quad $ and } \\
M(g,s) &:= M(s).
\end{align*}
This is well-defined because $B_n$ is compact.

\vspace{1ex}
Now assume that $\Phi([\varphi]) = 0$ for a morphism $\varphi$ of $\prod_{n \in \bbN}^\bd \! \calO_{G_n}^\lf(E_n; \calA_n)$, i.e.\ that there is a diagram \vspace{-1em}
\begin{center}
\begin{tikzcd}[row sep = 1em, column sep = 1ex]
(\coprod_n S_n/G_n, \cop_n p_n \comp \pi_n, \cop_n M_n)_n \arrow[dr] \arrow[rr, "\sum_{h \in G_n} (\varphi_n)^-_{h-}"] && (\coprod_n S'_n/G_n, \cop_n p_n \comp \pi'_n, \cop_n M'_n)_n \\
& (S, \pi, M) \arrow[ur]
\end{tikzcd}
\end{center}
where $(S, \pi, M) \in \calO(\coprod_{n \in \bbN} B_n; \calA_n)$.

Choose $R > 0$ such that $\varphi$ is $R$-controlled. By asymptotic faithfulness, there is an $N_0 \in \bbN$ such that for all $n \geq N_0$ the map $p_n$ is a $2R$-metric cover. 
Also, there is an $N_1 \in \bbN$ such that $\im(\pi) \subseteq ( \coprod_{n \leq N_1} B_n) \times \bbN$. Set $N := \max \{N_0, N_1\}$.
Since the diagram commutes, we have
$$0 = \sum_{h \in G_n} (\varphi_n)^-_{h-} \quad \text{ for all } n \geq N.$$
Now let $n \geq N$ and assume $\varphi^s_{s'} \neq 0$ for some $s,s' \in S_n$. 
Then the control conditions assert that $d(\pi_{E_n}(s), \pi'_{E_n}(s')) < R$. Since $p_n$ is a $2R$-metric cover, the non-trivial elements $h \in G_n$ translate elements in $E_n$ over large distances. More precisely, by the triangle inequality,
$$d(\pi_{E_n}(s), \pi'_{E_n}(hs')) \geq \underset{\geq 2R \text{ (Lemma \ref{lem:R_translative})}}{\underbrace{d(\pi_{E_n}(s'), \pi'_{E_n}(hs'))}} - \underset{< R}{\underbrace{d(\pi_{E_n}(s),\pi'_{E_n}(s'))}} > R$$
and therefore $(\varphi_n)^s_{hs} = 0$ for all $h \neq e$. We conclude
$$0 = \sum_{h \in G_n} (\varphi_n)^s_{hs'} = (\varphi_n)^s_{s'}.$$
Thus, $\varphi_n$ must be trivial for all $n \geq N$ and $(\varphi_n)_n$ factors through an object in $\bigoplus_{n \in \bbN} \calO_{G_n}(E_n; \calA_n)$. This implies that $\Phi$ is faithful.
\end{proof}

\subsection{Equivalent Formulations of the Conjecture}

\begin{lemma}
	\label{lem:conjecture_reduce}
	The following map is an isomorphism.
	$$\colim_{d \geq 0} H_*\big(P_d(X); \calA\big) \rightarrow K_*\big(\calC(X; \calA)\big)$$
\end{lemma}

\begin{proof} Consider the diagram
	\begin{center}
		\begin{tikzcd}
			\colim_{d \geq 0} H_*\big(P_d(X); \calA\big) \arrow[rr] \arrow[d, "\cong"] && K_*\big(\calC(X; \calA)\big) \arrow[d, "\cong"] \\
			H_*\big(\colim_{d \geq 0} P_d(X); \calA\big) \arrow[r, "\cong"] & H_*\big(\pt; \calA\big) \arrow[r, "\cong"] & K_*(\calA).
		\end{tikzcd}
	\end{center} \vspace{-2em}
\end{proof}

\begin{lemma}
	\label{lem:large_scale_assembly_map}
	Let $X$ be a locally compact metric space. 
	Conjecture \ref{conj:main_conjecture} is equivalent to
	$$\colim_{d \geq 0} K_*\big(\D^\ls(P_d(X); \calA)\big) \xrightarrow{\,\cong\,} K_{*-1}\big(\T^\ls(X; \calA)\big).$$
\end{lemma}

\begin{proof}
	The diagram 
	\begin{center}
		\begin{tikzcd}
			\T(X; \calA) \arrow[r]\arrow[d] & \O(X; \calA) \arrow[r]\arrow[d] & \D(X; \calA) \arrow[d] \\
			\T^\lf(X; \calA) \arrow[r]\arrow[d] & \O^\lf(X; \calA) \arrow[r]\arrow[d] & \D^\lf(X; \calA) \arrow[d] \\
			\T^\gg(X; \calA) \arrow[r] & \O^\gg(X; \calA) \arrow[r] & \D^\gg(X; \calA)
		\end{tikzcd}
	\end{center}
	induces a grid diagram in $K$-theory, where each row and column is a long exact Karoubi sequence, see Lemma \ref{lem:karoubi_filtration}. 
	Note that $\D^\gg(X; \calA)$ can either be defined as $\bigslant{\calD^\lf(X; \calA)}{\calD(X; \calA)}$ or as $\bigslant{\calO^\gg(X; \calA)}{\calT^\gg(X; \calA)}$.
	Both definitions yield equivalent categories.
	
	We are particularly interested in the following segment.
	\begin{center}
		\begin{tikzcd}[column sep=0em]
			& \dots \arrow[d] && \dots \arrow[d] \\
			H_*(X; \calA) \arrow[r, equals] & K_*\big(\calD(X; \calA)\big) \arrow[rr, "\partial"]\arrow[d] &\quad& K_{*}\big(\T(X; \calA)\big) \arrow[d] \\
			H_*^{\lf}(X; \calA) \arrow[r, equals] & K_*\big(\calD^\lf(X; \calA)\big) \arrow[rr, "\partial"]\arrow[d] && K_{*}\big(\T^\lf(X; \calA)\big) \arrow[d] \\
			&K_{*+1}\left(\D^\gg(X; \calA)\right) \arrow[rr, "\partial"] \arrow[d] && K_{*}\left(\T^\gg(X; \calA)\right) \arrow[d] \\
			& \dots && \dots
		\end{tikzcd}
	\end{center}
	Evidently, we can replace all $X$ by $P_d(X)$ and consider the colimit along $d$. (On the right hand side this has no effect since every flavour of $\calT$ is coarsely invariant.) Again, note that directed limits are exact.
	
	As $\colim_{d \geq 0} H_*\big(P_d(X); \calA\big) \rightarrow K_*\big(\calC(X; \calA)\big)$ (cf.\ upper row) is an isomorphism by Lemma~\ref{lem:conjecture_reduce}, the five-lemma tells us that the coarse assembly map (cf.\ middle row) is an isomorphism if and only if $\colim_{d \geq 0} K_*\big(\D^\ls(P_d(X); \calA)\big) \rightarrow K_{*-1}\big(\T^\ls(X; \calA)\big)$ (cf.\ lower row) is an isomorphism.
\end{proof}

\begin{corollary}
	\label{cor:third_term}
	Conjecture \ref{conj:main_conjecture} is equivalent to
	$$\colim_{d \geq 0} K_*\big(\O^\ls(P_d(X); \calA)\big) = 0.$$
\end{corollary}

\begin{proof}
	Consider the long exact sequence associated to the Karoubi filtration $\big(\calO^\ls, \calT^\ls\big)$ and apply $\colim_{d \geq 0}$. Note that directed colimits are exact. The resulting long exact sequence contains both the assembly map appearing in Lemma \ref{lem:large_scale_assembly_map} and the term $\colim_{d \geq 0} K_*\big(\O^\ls(P_d(X); \calA)\big)$.
\end{proof}

Consider an asymptotically faithful sequence of covers $p_n: E_n \rightarrow B_n$ with deck transformation groups $G_n$ and assume every $B_n$ to be compact.

\begin{proposition}
\label{prop:equivalent_conjecture}
Conjecture \ref{conj:main_conjecture} for $X = \bigsqcup_{n \in \bbN} B_n$ is equivalent to
$$\colim_{d \geq 0} K_*\bigg( \frac{\prod^\bd_{n\in \bbN}}{\bigoplus_{n \in \bbN}} \calO_{G_n}(P_d(E_n); \calA)\bigg) = 0.$$
\end{proposition}

\begin{proof}
This follows directly from Corollary \ref{cor:third_term} and Lemma \ref{lem:isomorphism_to_equivariant}.
\end{proof}

If we take a closer look at the statement of Proposition \ref{prop:equivalent_conjecture}, we realize that it allows for a slight, but useful, generalization.

\begin{proposition}
	\label{conj:main_technical_conjecture}
	Let $p_n: E_n \rightarrow B_n$ be an asymptotically faithful sequence of covers with deck transformation groups $G_n$ where every $B_n$ is compact.
	Then the following statements are equivalent.
	\begin{enumerate}
		\item $\colim_{d \geq 0} K_*\bigg( \frac{\prod^\bd_{n\in \bbN}}{\bigoplus_{n \in \bbN}} \calO_{G_n}(P_d(E_n); \calA)\bigg) = 0$ for all additive $G_n$-categories $\calA$.
		\item $\colim_{d \geq 0} K_*\bigg( \frac{\prod^\bd_{n\in \bbN}}{\bigoplus_{n \in \bbN}} \calO_{G_n}(P_d(E_n); \calA_n)\bigg) = 0$ for all sequences $\calA_n$ of additive $G_n$-categories.
	\end{enumerate}
	
\end{proposition}

\begin{proof}
	Clearly, Statement 2 implies Statement 1. For the other direction, consider the following section/retract
	$$\frac{\prod^\bd_{n\in \bbN}}{\bigoplus_{n \in \bbN}} \calO_{G_n}(P_d(E_n); \calA_n) \rightleftarrows \frac{\prod^\bd_{n\in \bbN}}{\bigoplus_{n \in \bbN}} \calO_{G_n}\left(P_d(E_n); \bigoplus\nolimits_{k \in \bbN} \calA_k\right)$$
	induced by the inclusions $\calA_n \rightarrow \bigoplus_k \calA_k$ and projections $\bigoplus_k \calA_k \rightarrow \calA_n$.
	This induces a section/retract on $K$-theory and the colimit, which shows that Statement 1 implies Statement 2.
\end{proof}

	\section{Residually Finite Groups}

It seems natural to state the following conjecture.

\begin{conjecture}
	\label{cor:coarse_conjecture_for_residually_finite_groups}
	Conjecture \ref{conj:main_conjecture} is true for box spaces of residually finite groups.
\end{conjecture}

In the following, we will prove that this conjecture is true for a residually finite group $G$ if the usual Farrell--Jones conjecture with coefficients is true for $G$, see Theorem \ref{thm:reduction_for_residually_finite_groups}.

\subsection{Virtually Cyclic Groups}

\begin{lemma}
Let $V$ be a residually finite group, $H_n \subseteq V$ the normal subgroups with $\bigcap_n H_n = \{1\}$, and $C \subseteq V$ a cyclic subgroup of finite index. Then the spaces
$$V/H_n \quad \text{ and } \quad C/(C \cap H_n)$$
are quasi isometric where the parameters only depend on $C$.
\end{lemma}

\begin{proof}
Using the {2\nd} isomorphism theorem, the space on the right hand side can be written as $CH_n / H_n$, which is a subspace of the left hand space. Now we have to find a global constant $K$ such that every point of $V/H_n$ is within distance $K$ to a point in $CH_n/H_n$. We have $[V : C] < \infty$.
Let $g_1, \dots, g_{[V: C]}$ be representatives for the cosets of $V/C$. Then a point $g \in V$ can be expressed as $xg_s$ for some $x \in C$ and $1 \leq s \leq [V: C]$ and we have
\begin{equation*}
d(g, CH_n) \leq d(g, C) \leq d(g,x) \leq |g_s| \leq \max\{|g_1|, \dots, |g_{[V: C]}|\} =: K < \infty. \qedhere
\end{equation*} 
\end{proof}

\begin{corollary}
\label{cor:coarse_conjecture_for_virtually_cyclic_groups}
The asymptotically faithful sequences associated to virtually cyclic, residually finite groups satisfy Conjecture \ref{conj:main_conjecture}.
\end{corollary}

\begin{proof}
	Let $V$ be a residually finite and virtually cyclic group. 
	If $V$ is finite, then $V/H_n$ is quasi isometric to $\pt$. 
	Otherwise, $V/H_n$ is quasi isometric to $C_n$, where $C_n$ are cyclic graphs whose perimeters (and therefore their girths) tend to infinity. 
	
	In both cases, the parameters of the quasi isometries do not depend on $n$. 
	The conjecture is then either equivalent that of $X = \bigsqcup_{n \in \bbN} \pt$ or that of $X = \bigsqcup_{n \in \bbN} C_n$.
	In both cases, $X$ has finite asymptotic dimension, hence finite decomposition complexity, and we again refer to \cite[Theorem~6.4]{Ramras.2018}.
\end{proof}

\subsection{Reduction for Residually Finite Groups}

The following Theorem and parts of its proof were inspired by \cite[Theorem 4.17]{Oyono-Oyono.2009}.

\begin{theorem}
	\label{thm:reduction_for_residually_finite_groups}
	Let $G$ be a residually finite group with normal subgroups $H_n$ such that $\bigcap_{n \in \bbN} H_n = \{1\}$. Set $\calA_G := \frac{\prod_n}{\bigoplus_n} \calC_0(G/H_n; \calA_n)$.
	Note that this still depends on the choice of the $H_n$.
	The category $\calA_G$ has a $G$-action given by $g.(S, \pi, M) := (S, g \comp \pi, g \comp M)$.
	
	The following statements are equivalent.
	\begin{enumerate}
		\item $\colim_{d \geq 0} K_*\bigg(\dfrac{\prod^\bd_n}{\bigoplus_n}\calO_{H_n}(P_d(G); \calA_n)\bigg) = 0$.
		\item The assembly map
		$H_*^G(E_\Fin G; \calA_G) \rightarrow K_*(\calA_G[G])$
		is an isomorphism.
		\item The assembly map
		$H_*^G(E_\VCyc G; \calA_G) \rightarrow K_*(\calA_G[G])$
		is an isomorphism.
	\end{enumerate}
\end{theorem}

\begin{remark}
	$\Fin$ is the family of finite subgroups of $G$.
	It is a fact that $\colim_{d \geq 0} P_d(G)$ is a model for $E_\Fin G$.
\end{remark}

\begin{proof}[Proof of Theorem \ref{thm:reduction_for_residually_finite_groups}]
	By Lemma \ref{lem:induction} and Lemma \ref{lem:products_and_coefficients} (see below) we have
	\begin{align*}
	\frac{\prod^\bd_n}{\bigoplus_n} \O_{H_n}(P_d(G); \calA_n)
	&\simeq \frac{\prod^\bd_n}{\bigoplus_n} \O_{G}(P_d(G); \calC_0(G/H_n;\calA_n)) \\
	&\simeq \O_{G}\bigg(P_d(G) ; \underset{\calA_G}{\underbrace{\frac{\prod_n}{\bigoplus_n} \calC_0(G/H_n; \calA_n)}}\bigg).
	\end{align*}
	This shows that
	$$\colim_{d \geq 0} K_*\bigg(\frac{\prod^\bd_n}{\bigoplus_n}\calO_{H_n}(P_d(G); \calA_n)\bigg) = 0$$
	is equivalent to
	$$K_*\big(\O_{G}(E_\Fin G ; \calA_G) \big) = \colim_{d \geq 0} K_*\big(\O_{G}(P_d(G) ; \calA_G) \big) = 0$$
	which is exactly the obstruction for the assembly map
	$$H_*^G(E_\Fin G; \calA_G) \rightarrow K_*(\calA_G[G])$$
	being an isomorphism. 
	
	\vspace{1ex}
	Statements 2.\ and 3.\ being equivalent follows from Lemma \ref{lem:no_nil_terms}. The proof uses the equivalence of 1.\ and 2.\ without causing the proof to be circular.
\end{proof}

\begin{lemma}
	\label{lem:no_nil_terms}
	In the situation of Theorem \ref{thm:reduction_for_residually_finite_groups}, the canonical map $E_\Fin G \rightarrow E_\VCyc G$ induces an isomorphism
	$$H_*^G(E_\Fin G; \calA_G) \rightarrow H_*^G(E_\VCyc G; \calA_G).$$
\end{lemma}

\begin{proof}
	Using the transitivity principle (cf.\ \cite[Proposition 2.9]{Lueck.2005}) we only have to show that the assembly maps 
	\begin{equation}
	H_*^V(E_\Fin V; \calA_G) \rightarrow K_*(\calA_G[V]) \tag{$\ast$}
	\end{equation}
	are isomorphisms for all virtually cyclic subgroups $V \subseteq G$. So let $V \subseteq G$ be virtually cyclic. Since it is a subgroup of a residually finite group, $V$ itself is residually finite, for which we can take $V \cap H_n \subseteq V$ as the normal subgroups. Clearly, $\bigcap_n V \cap H_n = \{1\}$.
	
	At the end of this proof we will show that there is an equivalence of categories $\calA_G \rightarrow \calB_V$ (with $V$-action) where
	$$\calB_V = \frac{\prod_n}{\bigoplus_n} \calC_0(V/(V \cap H_n); \calB_n) \quad \text{ and } \quad \calB_n := \calC_0(G/V H_n; \calA_n).$$
	If we know this, then ($\ast$) is an isomorphism if and only if
	$$ H_*^V(E_\Fin V; \calB_V) \rightarrow K_*(\calB_V[V]). $$
	However, by Theorem \ref{thm:reduction_for_residually_finite_groups}, this is the case if and only if 
	$$\colim_{d \geq 0} K_*\bigg( \frac{\prod^\bd_{n\in \bbN}}{\bigoplus_{n \in \bbN}} \calO_{V \cap H_n}(P_d(V); \calB_n)\bigg) = 0.$$
	Here again, we know this to be true because of Corollary \ref{cor:coarse_conjecture_for_virtually_cyclic_groups}.
	
	For the equivalence $\calA_G \rightarrow \calB_V$, regard the following chain of equivalences.
	\begin{align*}
	A_G &=\frac{\prod_n}{\bigoplus_n} \calC_0(G/H_n; \calA_n) \\
	&\cong \frac{\prod_n}{\bigoplus_n} \calC_0(V H_n/H_n \times G/V H_n; \calA_n) \tag{Lemma \ref{lem:normal_by_normal}}  \\
	&\simeq \frac{\prod_n}{\bigoplus_n} \calC_0(V H_n/H_n; \calC_0(G/V H_n; \calA_n))\\
	&\cong \frac{\prod_n}{\bigoplus_n} \calC_0(V/(V \cap H_n); \calB_n) \tag{2$^\text{nd}$ isomorphism theorem} \\
	&= \calB_V \tag*{\qedhere}
	\end{align*}
\end{proof}

\begin{lemma}
	\label{lem:normal_by_normal}
	Let $G$ be a group with a normal subgroup $H \subseteq G$.
	Let $V \subseteq G$ be any subgroup.
	Then $VH/H \times G/VH \cong G/H$ as $V$-sets, in which we let $V$ act trivially on $G/VH$.
\end{lemma}

\begin{proof}
	\def\inv{{\operatorname{inv}}}
	Choose a section (of sets) $s: G/VH \rightarrow G$ of $\pr \comp \inv :G \rightarrow G/VH$ and set
	\begin{align*}
		\varphi: VH/H \times G/VH &\rightarrow G/H \\
		(vH, gVH) &\mapsto vs(gVH)H.
	\end{align*}
	Clearly, this is an $V$-map.
	The inverse to this map is given by
	\begin{align*}
		\psi: G/H &\rightarrow VH/H \times G/VH \\
		gH &\mapsto (gs(g^{-1}VH)^{-1}H, g^{-1}VH).
	\end{align*}
	This clearly is an $V$-map as well.
	
	The fact that $s$ is a section of $\pr \comp \inv$ translates to $s(gVH)^{-1}VH = gVH$ for all $g \in G$.
	The following computations verify that $\varphi$ and $\psi$ are indeed inverse to each other.
	\begin{align*}
	\varphi \comp \psi (gH) 
	&= \varphi\big((gs(g^{-1}VH)^{-1}H, g^{-1}VH)\big) \\
	&= gs(g^{-1}VH)^{-1}s(g^{-1}VH)H \\
	&= gH \\
	~\\
	\psi \comp \varphi (vH, gVH) 
	&= \psi\big(vs(gVH)H)\big) \\
	&= (vs(gVH)s(\underset{=s(gVH)^{-1}VH = gVH}{\underbrace{(vs(gVH))^{-1}VH}})^{-1}H,
	\underset{=s(gVH)^{-1}VH = gVH}{\underbrace{(vs(gVH))^{-1}VH}}) \\
	&= (vs(gVH)s(gVH)^{-1}H,gVH) \\
	&= (vH,gVH) \qedhere
	\end{align*}
\end{proof}

%
%
%

\begin{lemma}
	\label{lem:induction}
	Let $X$ be a locally compact metric space with isometric $G$-action and $H \subseteq G$ a subgroup.
	There is an isometry of $\varepsilon$-filtered categories
	$$\calO_G(X; \calC_0(G/H; \calA)) \simeq \calO_{H}(X; \calA),$$
	i.e.\ not only are those categories equivalent, but the $\varepsilon$-filtration is preserved as well.
\end{lemma}

\begin{proof}
For an object $(S, \pi, M) \in \calO_{G}(X; \calC_0(G/H; \calA))$ denote
$(S_s, \pi_s, M_s) := M(s) \in \calC_0(G/H; \calA)$.
We have $M(gs) = g.M(s)$ implying $S_{gs} = S_s$, $\pi_{gs} = g.\pi_s$ and $M_{gs} = g.M_s$.
Using this notation we define a functor
$$\overline F: (S, \pi, M) \mapsto (\overline{S}, \overline{\pi}, \overline{M})$$
via 
\begin{align*}
	\overline{S} &:= \coprod\nolimits_{s \in S} \pi_s^{-1}(H/H) \quad \text{ with $H$-action } \quad h.[t \in S_s] := [t \in S_{hs}], \\
	\overline{\pi} &: [t \in S_s] \mapsto \pi(s), \text{ and}\\
	\overline{M} &: [t \in S_s] \mapsto M_s(t).
\end{align*}
The free $G$-action on $S$ restricts to a free $H$-action on $\overline{S}$.
Also $\overline{\pi}$ and $\overline{M}$ are $H$-equivariant because we have $\pi(hs) = h.\pi(s)$ and $M_{hs} = h.M_s$ for all $h \in H$.
We define $\overline F$ on morphisms simply by
$$\big[\varphi : (S, \pi, M) \rightarrow (S', \pi', M')\big] \mapsto \big[\overline \varphi: (\overline S, \overline \pi, \overline M) \rightarrow (\overline S', \overline \pi', \overline M')\big]$$
where $\overline \varphi^{[t \in S_s]}_{[t' \in S'_{s'}]} := (\varphi^s_{s'})^t_{t'} : M_s(t) \rightarrow M'_{s'}(t')$.
As this is merely a rearrangement of symbols, it is clear that the functor is fully faithful.

In order to see that $\overline{F}$ is essentially surjective, consider an object $(S, \pi, M) \in \calO_{H}(X; \calA)$.
We claim that this is isomorphic to the image of $(\widetilde{S}, \widetilde{\pi}, \widetilde{M})$, where
\begin{align*}
\widetilde S &:= G \times_H S := \bigslant{G \times S}{(gh,s) \simeq (g,hs) \forall h \in H}, \\
\widetilde \pi&: [g,s] \mapsto g.\pi(s), \text{ and } \\
\widetilde M&: [g,s] \mapsto (\{\ast\}, \ast \mapsto gH, \ast \mapsto g.M(s)).
\end{align*}
It is easy to see that $\widetilde S$ is indeed a free $G$-set.
Also, an immediate calculation shows that $\widetilde \pi$ and $\widetilde M$ are well-defined.
Now $\overline{F}\big(\widetilde{S}, \widetilde{\pi}, \widetilde{M}\big)$ is given by $\big(\overline{\widetilde{S}}, \overline{\widetilde{\pi}}, \overline{\widetilde{M}}\big)$.

We see
$$\overline{\widetilde S} = \coprod_{[g,s] \in G \times_HS} \underset{= \{\ast\} \text{ if } g \in H}{\underbrace{\widetilde\pi_s^{-1}(H/H)}} \cong H \times_H S \cong S.$$
Therefore, every element in $\overline{\widetilde S}$ can be written as $[1,s]$, for $s \in S$. This gives
\begin{align*}
\overline{\widetilde \pi}: [\ast \in \widetilde\pi_s^{-1}(H/H)] \text{ of } [1,s]\text{-component } &\mapsto \;\widetilde{\pi}([1,s]) = 1.\pi(s) = \pi(s) \\
\overline{\widetilde{M}} : [\ast \in \widetilde\pi_s^{-1}(H/H)] \text{ of } [1,s]\text{-component } &\mapsto \;1.M(s) = M(s). \qedhere
\end{align*}
\end{proof}

\begin{lemma}
\label{lem:products_and_coefficients}
Let $X$ be a locally compact metric space of bounded geometry with isometric $G$-action and $\calA_n$ a sequence of additive categories.
There is an equivalence of categories
$$\Phi: \calO_G\bigg(X; \frac{\prod_n}{\bigoplus_n} \calA_n\bigg) \rightarrow \frac{\prod^\bd_n}{\bigoplus_n} \calO_G(X; \calA_n).$$
\end{lemma}

\begin{proof}
	Regard an object $(S, \pi, M) \in \calO_G \left( X ; \frac{\prod_n}{\bigoplus_n} \calA_n \right)$.
	The image of $(S, \pi, M)$ under $\Phi$ is the object $(S_n, \pi_n, M_n)_n \in \dfrac{\prod^\bd_n}{\bigoplus_n} \calO_{G} \big(X; \calA_n \big)$, which we define via $S_n := S$, $\pi_n := \pi$ and $M_n := \pr_n \comp M$.
	
	A morphism $\varphi: (S, \pi, M) \rightarrow (S', \pi', M')$ is mapped to a sequence of morphisms $\varphi_n: (S_n, \pi_n, M_n) \rightarrow (S_n', \pi_n', M_n')$ simply given by restrictions, i.e. $(\varphi_n)^s_{s'} := \pr_n \comp \varphi^s_{s'} \comp \iota_n$.
	This functor clearly is fully faithful.
	
	Let us now verify essential surjectivity. To do so, we choose a zero-convergent sequence $(\delta_k)_{k \in \bbN}$ and a sequence $X_k$ of maximal $\delta_k$-separated $G$-subspaces of $X$ with maps $f_k : X \rightarrow X_k$ such that $d(x, f_k(x)) \leq \delta_k$. 
	Now let $(S_n, \pi_n, M_n)_n \in \frac{\prod^\bd_n}{\bigoplus_n} \calO_{G} \big(X; \calA_n \big)$.
	Define $(S, \pi, M) \in \calO_G\bigg(X; \frac{\prod_n}{\bigoplus_n} \calA_n\bigg)$ via
	\begin{align*}
		S &:= \coprod_{k \in \bbN} X_k, \\
		\pi &: [x \in X_k] \mapsto (x, k) \in X \times \bbN, \text{ and} \\
		M &: [x \in X_k] \mapsto \left( \bigoplus \set{M_n(s)}{\pi_{n,\bbN}(s) = k \text{ and } f_k(\pi_{n,X}(s)) = x} \right)_n \in \frac{\prod_n}{\bigoplus_n} \calA_n.
	\end{align*}
	This way, we still have all modules $M_n(s)$, but we placed them in a locally finite way (cf. Property \ref{c_locally_compact} of Definition \ref{def:controlled_categories}).
	We define an isomorphism $(\varphi_n)_n : (S_n, \pi_n, M_n)_n \rightarrow \Phi(S,\pi,M)$ via
	$$ (\varphi_n)^s_{[x \in X_k]} : M_n(s) \rightarrow \bigoplus \set{M_n(t)}{\pi_{n,\bbN}(t) = k \text{ and } f_k(\pi_{n,X}(t)) = x},$$
	which is the inclusion if $M_n(s)$ is a direct summand of the right hand side, or 0 otherwise.
	We can do the same thing with projections in the other direction.
	This shows that $(\varphi_n)_n$ is indeed an isomorphism.
	Yet we still need to check the control condition.
	
	Assume $(\varphi_n)^s_{[x \in X_k]} \neq 0$.
	Then $\pi_{n,\bbN}(s) = k$ and $f_k(\pi_{n,X}(s)) = x$.
	The defining property of $f_k$ now ensures
	\begin{equation*}
		d(\pi_{n,X}(s), x) = d(\pi_{n,X}(s), f_k(\pi_{n,X}(s))) \leq \delta_k. \qedhere
	\end{equation*}
\end{proof}
	\section{Gromov Monsters}

Gromov monster groups are groups that weakly contain expanders.
These groups yield counter-examples for the Baum--Connes conjecture, see \cite[Theorem~8.2]{Willett.2012a}.
As Osajda points out in \cite[Section 1.3]{Osajda.2020}, a weak embedding is not necessarily a coarse embedding and that having a coarse embedding is crucial for some results, as in Willett and Yu's situation.
In the same article Osajda constructs finitely  generated groups with expanders isometrically embedded into their Cayley graphs (cf.\ \cite[Theorem 4]{Osajda.2020}).
The necessary ingredient is a sequence of expanders of large girth and bounded diameter-by-girth ratio.

The easiest way to create such a sequence is to consider the residually finite group $SL_2(\bbZ)$.
It is a well-known fact that this group contains a free group of rank two, generated by $A = \begin{psmallmatrix}
	1 & 2 \\ 0 & 1
\end{psmallmatrix}$ and $B = \begin{psmallmatrix}
1 & 0 \\ 2 & 1
\end{psmallmatrix}$.
Margulis showed in \cite{Margulis.1982} that the Cayley graphs $\Gamma_p := \Cay_{\{A,B\}}(SL_2(\bbF_p))$ have logarithmic girth.
The precise statement is that $0.756 \cdot \log(|\Gamma_p|)$ is an asymptotic lower estimate for $\girth(\Gamma_p)$, cf. \cite[Section 6]{Margulis.1982}.

Using the facts that the sequence $(\Gamma_p)_{p \text{ prime}}$ is an expander 
and that expanders are of logarithmic diameter, we conclude that the $\Gamma_p$ have a bounded diameter-by-girth ratio.
These considerations have been generalized in \cite{Arzhantseva.2018}.

As the Farrell--Jones conjecture is true for free groups, in particular for $\langle A, B \rangle$, our main theorem shows that the bounded isomorphism conjecture is true for $\bigsqcup_{p} \Gamma_p$ as well.
We conclude that not all large girth expanders and monster groups causing trouble in the setting of the Baum--Connes conjecture cause trouble in the Farrell--Jones setting.
In particular, there cannot be a direct analogue of Willett and Yu's non-surjectivity result for the coarse Farrell--Jones assembly map. 
	
	\bibliography{bibliography.bib}
	
\end{document}